\documentclass[12pt]{article}
\usepackage{amsfonts}
\usepackage{mathrsfs,dsfont,bbm}

\usepackage[latin1]{inputenc}
\usepackage{amsmath,amssymb}
\usepackage{latexsym}
\usepackage{epstopdf}
\usepackage{geometry}
\usepackage[active]{srcltx}
\usepackage[
bookmarks=true,         
bookmarksnumbered=true, 
colorlinks=true, pdfstartview=FitV, linkcolor=blue, citecolor=blue,
urlcolor=blue]{hyperref}

\usepackage{graphicx}
\usepackage{subfig}
\newtheorem{theorem}{Theorem}[section]

\newtheorem{lemma}[theorem]{Lemma}

\newtheorem{remark}[theorem]{Remark}

\numberwithin{equation}{section}
\newenvironment{proof}[1][Proof]{\noindent\textit{#1.} }{\hfill \rule{0.5em}{0.5em}}

\begin{document}

\title{Exponential Euler method for stiff SDEs driven by fractional Brownian motion}
\date{\today}

\author
{Haozhe Chen, Zhaotong Shen, Qian Yu
\thanks{ School of Mathematics, Nanjing University of Aeronautics and Astronautics, Nanjing 211106, China
Corresponding author: qyumath@163.com}}

\maketitle

\begin{abstract}
\noindent In a recent paper by Kamrani \emph{et al.} \cite{Kam24}, exponential Euler method for stiff stochastic differential equations with additive fractional Brownian noise was discussed, and the convergence order close to the Hurst parameter $H$ was proved. Utilizing the technique of Malliavin derivative, we prove the exponential Euler scheme and obtain a convergence order of one, which is the optimal rate given by Kamrani \emph{et al.} \cite{Kam24} in numerical simulation.

\end{abstract}

\section{Introduction}\label{sec1}
In this paper, we consider the numerical approximation of a system of $m$-dimensional SDEs with a linear stiff term and driven by fractional Brownian motion (fBm)
\begin{equation}\label{sec1-eq.1}
\text{d}U_t=(AU_t+f(U_t))\text{d}t+\sum_{i=1}^Mb_i(t)\text{d}B_i^{H}(t), U_{t_0}=u_0\in\mathbb{R}^m, \qquad t\in[t_0, T],
\end{equation}
where $B^{H}_i, i=1,2,\cdots,M$ are independent $m$-dimensional fBm with Hurst parameter $H\in (1/2,1)$. We assume that the function $f: \mathbb{R}^m\to\mathbb{R}^m$ satisfies Lipschitz condition and the first three derivatives of $f$ satisfy the polynomial growth condition, $A\in \mathbb{R}^{m,m}$ and $AU_t$ is stiff:
$$
|A|(T-t_0)\gg 1, ~~\mu[A](T-t_0)\ll |A|(T-t_0),
$$
where $|\cdot|$ denotes the Euclidean norm, $\mu(\cdot)$ is the corresponding logarithmic matrix norm (see in \cite{Dek84}, \cite{Str12}).
The equation \eqref{sec1-eq.1} is considered as a pathwise Riemann-Stieltjes integral equation and the existence of a unique stationary and attracting solution with Hurst parameter $H>1/2$ follows from \cite{Mas04}.

Recently, Kamrani \emph{et al.} \cite{Kam24} studied the exponential Euler scheme to the SDE \eqref{sec1-eq.1} and proved the convergence order close to the Hurst parameter $H$.
\begin{theorem}\label{sec1-thm.1}\cite{Kam24}
Assume $\mathbb{E}|u_0|^2<\infty, \mathbb{E}|Au_0|^2<\infty$, function $f$ satisfies the Lipschitz condition and linear growth condition. Then, under some additional assumptions of $A$, for every $\varepsilon>0$, there exists a constant $C>0$ such that
$$
\sup_{k=0,1,\cdots,N}\sqrt{\mathbb{E}|U_{t_k}-V_{k}|^2}\leq C \left(\max_k\{t_{k+1}-t_k\}\right)^{H-\varepsilon},
$$
for all $N\geq 2$, where $U_{t_k}$ is the solution of SDE \eqref{sec1-eq.1} at time $t_k$, $V_{k}$ is the numerical solution
$$
V_{k+1}=e^{A(t_{k+1}-t_k)}V_k+A^{-1}(e^{A(t_{k+1}-t_k)}-Id)f(V_k)+\sum_{i=1}^M\int_{t_k}^{t_{k+1}}e^{A(t_{k+1}-s)}b_i(s)\text{d}B_i^{H}(s),
$$
for all $k=0,1,\cdots, N$ and $Id$ is the identity matrix.
\end{theorem}

Note that, the choice of $A$ has significant influence on the long time integration error. In the situation that there is no noise, exponential integrators have been proven to be a very interesting class of numerical time integration methods (see in \cite{Hoc10}), as they handle the stiff part exactly and can thus be stable even when being explicit. Therefore, it is meaningful for Kamrani \emph{et al.} \cite{Kam24} to consider fBm as a noise study exponential Euler scheme.
Considering the construction conditions of matrix $A$, it is necessary to avoid the existence of $|A|$ in any error bound in practical proofs.

If $A=0$, Garrido-Atienza \cite{Gar09} proved that under a one-sided dissipative condition and any Hurst parameter $H\in(0,1)$ the drift-implicit Euler method have a unique stationary solution.
The numerical approximation for fractional SDE with $A=0$ has received much attention in Hu \emph{et al.}\cite{Hu16}, Hong \emph{et al.}\cite{Ho20}, Kloeden \emph{et al. }\cite{Klo11}, Neuenkirch\cite{Ne06}, Zhang and Yuan \cite{ZY21} and references therein.
Recently, Zhou \emph{et al.}\cite{Zhou23} have considered the backward Euler method and obtained the optimal convergence rate. This provides us with the basis for studying the optimal convergence rate of exponential Euler scheme for stiff type SDE when A is not equal to zero.

In this paper, we focus on exponential Euler method for stiff SDEs and prove the optimal convergence rate under the following assumptions on functions $f:  \mathbb{R}^m\to \mathbb{R}^m$, $b_i: [t_0, T]\to\mathbb{R}^m$ and the matrix $A$:
\begin{enumerate}
  \item[(A1)] There exists a constant $\kappa$, such that for any $x,y\in\mathbb{R}^m$,
$$|f(x)|\leq \kappa(1+|x|)$$
and
$$|f(x)-f(y)|\leq \kappa|x-y|,$$
where $|\cdot|$ denotes the Euclidean norm.

  \item[(A2)] There are constants $\kappa$ and $\nu$, such that for any $x\in\mathbb{R}^m$,
$$\max_{i=1,2,3}|\partial^i f(x)|\leq (1+|x|^{\nu}),$$
where $|\partial^i f(x)|, i=1,2,3$ are defined as follows,
 $$|\partial f(x)|=\Big(\sum_{j=1}^m|\frac{\partial f}{\partial x_j}(x)|^2\Big)^{1/2},$$

$$|\partial^2 f(x)|=\Big(\sum_{j,k=1}^m|\frac{\partial^2 f}{\partial x_j\partial x_k}(x)|^2\Big)^{1/2},$$
and
$$|\partial^3 f(x)|=\Big(\sum_{j,k,\ell=1}^m|\frac{\partial^3 f}{\partial x_j\partial x_k\partial x_{\ell}}(x)|^2\Big)^{1/2}.$$

  \item[(A3)] The functions $b_i: [t_0, T]\to\mathbb{R}^m, i=1,2,\ldots,n$ are bounded with respect to the Euclidean norm, i.e. there exists a constant $C$, for any $t\in[t_0, T]$,
  $$|b_i(t)|\leq C, ~~i=1,2,\ldots,M.$$

\item[(A4)] There  exists a constant $L$ such that $A\in\mathbb{R}^{m,m}$ for all $s<t$ satisfies the following conditions
$$|Ae^{A(t-s)}|\leq \frac{L}{t-s}, ~~|A^{-1}(Id-e^{A(t-s)})|\leq L(t-s).$$
Furthermore, $\sup_{t_0\leq t\leq T}e^{\mu[A](t-t_0)}\leq C$ with a constant $C$ not too large, where $\mu[A]$ denotes the logarithmic matrix norm for any square matrix $A$ defined by
$$\mu[A]=\lambda_{max}\Big(\frac{A+A^T}{2}\Big).$$
\end{enumerate}
Note that, by \cite{Dek84}, the logarithmic matrix norm satisfies $$|\mu[A]|\leq |A| ~~\text{and} ~~|e^{A t}|\leq e^{\mu[A]t},$$ for all $t>0$. In applications, we often let $\mu[A]\leq0$, which implies that
\begin{equation}\label{sec1-eq-muA}
\sup_{t_0\leq t\leq T}e^{\mu[A](t-t_0)}\leq 1.
\end{equation}

\begin{remark}
Assumption (A1) here is a common assumption, in order to ensure the existence and uniqueness of SDE \eqref{sec1-eq.1}.
Compared to \cite{Kam24}, assumption (A2) is a new addition to this paper, which aims to use the Malliavin derivative to ensure the validity of Lemma \ref{sec3-lem.3}, which plays a crucial role in the proof of the main result.
Assumption (A4) may not appear intuitive enough, but it is fulfilled for finite dimensional case (see in \cite{Lor14}).
\end{remark}

Let $\pi: t_0<t_1<\cdots<t_N=T$ be a partition of the time interval $[t_0, T]$. We denote by non-equidistant sizes $h_i=t_{i+1}-t_i, 0\leq i\leq N-1$, and denote $h_{max}=\max_{0\leq i\leq N-1}h_i$.
The  classical Euler scheme applied to \eqref{sec1-eq.1} is
\begin{equation}\label{sec1-eq-BEE}
V_{k+1}=e^{A(t_{k+1}-t_0)}u_0+\sum_{j=0}^k\int_{t_j}^{t_{j+1}}e^{A(t_{k+1}-s)}f(V_{j})ds+\sum_{i=1}^M\int_{0}^{t_{k+1}}e^{A(t_{k+1}-s)}b_i(s)dB_i^{H}(s),
\end{equation}
where $V_k:=V_{t_k}$, for all $k=0,1,\cdots, N-1$ and $V_0=u_0$.

Next we present the main result of this paper.
\begin{theorem}\label{sec1-thm.2}
Assume $\mathbb{E}|u_0|^2<\infty, \mathbb{E}|Au_0|^2<\infty$, the functions $f, b_i, i=1,2,\ldots,M$, and matrix $A$ satisfy the assumptions (A1)--(A4). Then, there exists a constant $C>0$ such that
$$
\sup_{k=0,1,\cdots,N}\sqrt{\mathbb{E}|U_{t_k}-V_{k}|^2}\leq C \max_k\,h_k,
$$
for all $N\geq 2$, where $U_{t_k}$ is the solution of SDE \eqref{sec1-eq.1} at time $t_k$.
\end{theorem}

\begin{remark}
The order of the convergence rate in the above theorem is one. We have verified in numerical experiment that this is the optimal rate,
which is consistent with the results that Kamrani \emph{et al.} \cite{Kam24} want to improve.
\end{remark}

The paper is organized as follows. In Section 2 we recall the definition of fBm, the techniques of Malliavin calculus and some preliminary lemmas. Section 3 is devoted to proving the main result.
Then in Section 4, we report our numerical experiment and illustrate the accuracy of the method.
Throughout this paper, if not mentioned otherwise, $C$ or $c$ (maybe with subscript) will denote a generic positive finite constant and may change from line to line.

\section{Preliminaries}\label{sec2}
\subsection{Fractional Brownian motion}\label{sec2.1}
In this subsection we will recall the definition of fBm (see in \cite{BHOZ2008}, \cite{Hu17}).
FBm on $\mathbb{R}^d$ with Hurst parameter $H\in(0,1)$ is a $d$-dimensional centered Gaussian process $B^H=\{B_t^H, ~t\geq0\}$  with component processes being independent copies of an one dimensional centered Gaussian process $B^{H,i}, i=1,2,\ldots,d$ and the covariance function given by
$$
\mathbb{E}[B_t^{H,i}B_s^{H,i}]=\frac{1}{2}\left[t^{2H}+s^{2H}-|t-s|^{2H}
\right].
$$
Note that $B_t^{\frac12}$ is a classical standard Brownian motion.  We only consider fBm with Hurst parameter $H>1/2$ in this paper.

FBm $B^H=(B^{H,1},\cdots,B^{H,d})$ admits the following Wiener integral representation
$$B_t^{H,i}=\int_0^tK_H(t,s)dW^i_s, i=1,2,\ldots,d,$$
where $W^i$ is an one dimensional standard Brownian motion and $K_H$ is the kernel function defined by
\begin{align}\label{sec2-eq.KH}
K_H(t,s)=C_Hs^{1/2-H}\int_s^t(u-s)^{H-3/2}u^{H-1/2}du, ~~s\leq t,
\end{align}
with $C_H=\sqrt{\frac{H(2H-1)}{\beta(2-2H,H-1/2)}}$. Note that
$$\frac{\partial K_H}{\partial t}(t,s)=c_H(H-\frac12)(t-s)^{H-3/2}(\frac{s}{t})^{1/2-H}.$$

Young proved that the integral $\int_{a}^{b}fdg$ exists
as a Riemann-Stieltjes integral if $f$ and $g$ have finite $p$-variation and $q$-variation, respectively, with $\frac1p+\frac1q>1$.
Clearly, if $f$ is $\alpha$-H\"older continuous, then it has finite $\frac1\alpha$-variation on
any finite interval. In this case, fBm $B^{H,i}$ has  H\"{o}lder continuous sample paths of exponent of order lesser than $H$,
the fractional stochastic integral
$$\int_0^tu(s)dB_s^{H,i}$$
exists as a pathwise Riemann-Stieltjes integral if $u$ is H\"{o}lder continuous of order greater than $1-H$.

The Cameron-Martin space $\mathcal{K}_H$ associated to the covariance $\mathbb{E}[B_t^{H,i}B_s^{H,i}]$ is defined as the closure of the space of step functions with respect to the scalar product
$$\langle \phi,\psi\rangle_{\mathcal{K}_H}=c_H\int_0^t\int_0^t\phi(u)\psi(v)|u-v|^{2H-2}dsdt.$$

For any $\phi\in \mathcal{K}_H$, we can define a operator $K^{*}$,
\begin{align*}
K^{*}_t(\phi)(s)=\int_s^t\phi(r)\frac{\partial K_H}{\partial r}(r,s)dr.
\end{align*}
Then we have an isometry between $\mathcal{K}_H$ and $L^2([0,t])$,
$$\langle \phi, \psi\rangle_{\mathcal{K}_H}=\langle K^{*}_t(\phi),K^{*}_t(\psi)\rangle_{L^2[0,t]}.$$
Thus, the Wiener integral with respect to fBm $B^{H,i}$ can be rewritten as a Wiener integral with respect to Wiener process $W^i$
\begin{align*}
\int_0^t\phi(s)dB_s^{H,i}=\int_0^tK^{*}_t(\phi)(s)dW^i_s.
\end{align*}

\subsection{Malliavin calculus}\label{sec2.2}
Let $\{G_t, t\in[0,T]\}$ be a zero mean continuous Gaussian process with covariance function $\mathbb{E}[G_tG_s]=R(t,s)$ such that $G_0=0$.
We suppose that $G$ is defined in a complete probability space $(\Omega, \mathcal{F}, \mathbb{P})$ and $\mathcal{F}$ is generated by $G$. Let $\mathcal{H}_1$
be the first Wiener chaos, the closed subspace of $L^2(\Omega)$ generated by $G$. The reproducing kernel Hilbert space $\mathfrak{H}$ is defined as the closure
of the linear span of the indicator functions $\{\mathbf{1}_{[0,t]}, t\in[0,T]\}$ with respect to the scalar product
$\langle \mathbf{1}_{[0,t]}, \mathbf{1}_{[0,s]}\rangle_{\mathfrak{H}}$.
If $G$ is an  one dimensional fBm $B^{H,1}$,
$$\langle \mathbf{1}_{[0,t]}, \mathbf{1}_{[0,s]}\rangle_{\mathfrak{H}}=\mathbb{E}[(B_t^{H,1}-B_s^{H,1})(B_u^{H,1}-B_v^{H,1})].$$
The mapping $\mathbf{1}_{[0,t]}\mapsto G_t$ provides an isometry between $\mathfrak{H}$ and $\mathcal{H}_1$. We denote by $G(\varphi)$ the image in $\mathcal{H}_1$ of an element $\varphi\in\mathfrak{H}$.

For a smooth and cylindrical random variable $F=f(G(\varphi_1),\cdots,G(\varphi_n))$ with $\varphi_i\in\mathfrak{H}$,
$f\in C_b^\infty(\mathbb{R}^n)$ (i.e. $f$ and all of its partial derivatives are bounded), we define its Malliavin derivative as
the $\mathfrak{H}$-valued random variable given by
$$D^GF=\sum_{i=1}^n\frac{\partial f}{\partial x_i}(G(\varphi_1),\cdots,G(\varphi_n)){\varphi_i}.$$

By iteration, one can define the $k$-th derivative $D^{G,k}F$ as an element of $L^2(\Omega,\mathfrak{H}^{\otimes k})$, where $\mathfrak{H}^{\otimes k}$ denote the $k$-th tensor product of the Hilbert space $\mathfrak{H}$. For any natural
number $k$ and any real number $p\geq1$, we define the Sobolev space $\mathbb{D}_G^{k,p}$ as the closure of the space
of smooth and cylindrical random variables with respect to the norm $||\cdot||_{G,k,p}$ defined by
\begin{align}\label{sec2-eq.D-space}
||F||_{G,k,p}^p=\mathbb{E}(|F|^p)+\sum_{i=1}^k\mathbb{E}(||D^{G,i}F||^p_{\mathfrak{H}^{\otimes i}}).
\end{align}

The divergence operator $\delta^G$ is defined as the adjoint of the derivative operator $D^G$ in the following
manner. An element $u\in L^2(\Omega,\mathfrak{H})$ belongs to the domain of $\delta^G$, denoted by Dom $\delta^G$, if there is a
constant $C_u$ depending on $u$ such that
\begin{align}\label{sec2-eq.d-space}
|\mathbb{E}(\langle D^GF,u\rangle_{\mathfrak{H}})|\leq C_u ||F||_{L^2(\Omega)}, \forall F\in\mathbb{D}_G^{1,2}.
\end{align}
If $u\in$ Dom $\delta^G$, then the random variable $\delta^G(u)$ is defined by the duality relationship
$$\mathbb{E}(F\delta^G(u))=\mathbb{E}(\langle D^GF,u\rangle_{\mathfrak{H}}),$$
which holds for any $F\in\mathbb{D}_G^{1,2}$.

If $V$ is a separable Hilbert space, we can define in a similar way the spaces $\mathbb{D}_G^{k,p}(V)$
of $V$-valued random variables. We recall that the space $\mathbb{D}_G^{1,2}(\mathfrak{H})$ of $\mathfrak{H}$-valued random variables is included in the domain of $\delta^G$, and for any element $u\in\mathbb{D}_G^{1,2}(\mathfrak{H})$ we have
$$\mathbb{E}|\delta^G(u)|^2\leq \mathbb{E}\|u\|^2_{\mathfrak{H}}+\mathbb{E}\|D^Gu\|^2_{\mathfrak{H}^{\otimes 2}}.$$

Furthermore, Meyer inequalities imply that for all $p>1$, we have
$$\|\delta^G(u)\|_p\leq c_p\,\|u\|_{\mathbb{D}_G^{1,2}(\mathfrak{H})}.$$
If $u$ is a simple $\mathfrak{H}$-valued random variable of the form
$u=\sum_{i=1}^nF_i\varphi_i,$
where $F_i\in\mathbb{D}_G^{1,2}$ and $\varphi_i\in\mathfrak{H}$. Then $u\in$ Dom $\delta^G$ and
$$\delta^G(u)=\sum_{i=1}^n\left(F_iG(\varphi_i)-\langle D^GF_i,\varphi_i\rangle_{\mathfrak{H}}\right).$$
 For the convenience of writing, in the case of not causing confusion, we will simply write $D^G$, $\delta^G$ and $\mathbb{D}_G^{k,p}$ as $D$, $\delta$ and $\mathbb{D}^{k,p}$ (when $G$ is fBm) in the follow-up of this paper.
The reader is referred to \cite{Alos01,Nu06,Nu18} and references therein for more about Malliavin calculus.

\subsection{Preliminary lemmas}\label{sec2.3}
In this subsection, we provide two lemmas, and the results of these lemmas will be used for the proof in Section 3.
The first lemma about the upper-bound estimate result for the increments of fBm $B^H$ comes from the Proposition 4.2 in \cite{Zhou23}.
\begin{lemma}\label{sec2-lem-3-0}\cite{Zhou23}
Let $F: \Omega\to \mathbb{R}$ be a random variable that possesses the second Malliavin derivative. If $\mathbb{E}(|F|)<\infty$ and $\sup_{t_0\leq s\leq T}\mathbb{E}(D_s^{(i)}F)<\infty$,
then for any $t_0\leq s,t\leq T$
\begin{equation*}
\Big|\mathbb{E}\Big[F(B^{H,i}_t-B^{H,i}_s)\Big]\Big|\leq C|t-s|.
\end{equation*}
Assume further that $\sup_{t_0\leq s\leq T}\mathbb{E}(D_{s,t}^{(i,j)}F)<\infty$, then for any $t_0\leq s<t\leq T$ and $t_0\leq v<u\leq T$
\begin{equation*}
\Big|\mathbb{E}\Big[F(B^{H,i}_t-B^{H,i}_s)(B^{H,i}_u-B^{H,i}_v)\Big]\Big|\leq C|t-s||u-v|+C\langle \mathbf{1}_{[s,t]}, \mathbf{1}_{[v,u]}\rangle_{\mathfrak{H}},
\end{equation*}
where $\langle \mathbf{1}_{[s,t]}, \mathbf{1}_{[v,u]}\rangle_{\mathfrak{H}}=\mathbb{E}[(B_t^{H,i}-B_s^{H,i})(B_u^{H,i}-B_v^{H,i})]$ given in Section \ref{sec2}.
\end{lemma}

If we let $F:=F_1F_2$, where $F_1$ satisfies all the conditions of $F$ in the Lemma \ref{sec2-lem-3-0} and $F_2$ is a non random bounded function, then we have the following lemma.
\begin{lemma}\label{sec2-lem-3}
Let $F_1: \Omega\to \mathbb{R}$ be a random variable that possesses the second Malliavin derivative. If $\mathbb{E}(|F_1|)<\infty$ and $\sup_{t_0\leq s\leq T}\mathbb{E}(D_s^{(i)}F_1)<\infty$,
then for any continuously bounded function $g_0$,  $t_0\leq s,t\leq T$
\begin{equation}\label{sec3-eq.B-1}
\Big|\mathbb{E}\Big[F_1\int_s^tg_0(u)dB^{H,i}_u\Big]\Big|\leq C|t-s|.
\end{equation}
Assume further that $\sup_{t_0\leq s\leq T}\mathbb{E}(D_{s,t}^{(i,j)}F_1)<\infty$, then for any continuously bounded functions $g_1, g_2$, $t_0\leq s<t\leq T$ and $t_0\leq v<u\leq T$
\begin{equation}\label{sec3-eq.B-2}
\Big|\mathbb{E}\Big[F_1\int_s^tg_1(a)dB^{H,i}_a\int_v^ug_2(b)dB^{H,j}_b\Big]\Big|\leq C|t-s||u-v|+C\langle \mathbf{1}_{[s,t]}, \mathbf{1}_{[v,u]}\rangle_{\mathfrak{H}}.
\end{equation}
\end{lemma}
\begin{proof}
Because $g_j, j=0,1,2$ here are continuous functions, we only need to rewrite $\int_s^tg_j(u)dB^{H,i}_u, j=0,1,2$, to $F_2(B^{H,i}_t-B^{H,i}_s)$. Then by the boundedness and non randomness of $F_2$, Lemma \ref{sec2-lem-3} follows from Lemma \ref{sec2-lem-3-0}.
\end{proof}

\section{Proof of Theorem \ref{sec1-thm.1}}\label{sec3}

In this section, we will provide proof of the main result (Theorem \ref{sec1-thm.1}) of this paper.  Note that $B^{H}_i, i=1,2,\cdots,M$, are independent $m$-dimensional fBm, and the boundedness of function $b_i, i=1,2,\cdots,M$. Therefore, for the convenience, we only need to consider the case $M=1$. At this point, according to equations \eqref{sec1-eq.1}  and \eqref{sec1-eq-BEE}, the expressions for $U$ and $V$ can be rewritten as
\begin{equation}\label{sec3-eq.1}
U_t=e^{A(t-t_0)}u_0+\int_0^te^{A(t-s)}f(U_s)ds+\int_0^te^{A(t-s)}b(s)dB_s^{H},
\end{equation}
and for $t_0<t_1<\cdots<t_{k}<t<t_{k+1}$,
\begin{equation}\label{sec3-eq.2}
V_t=e^{A(t-t_0)}u_0+\sum_{j=0}^{k-1}\int_{t_j}^{t_{j+1}}e^{A(t-s)}f(V_{t_j})ds+\int_{t_k}^te^{A(t-s)}f(V_{t_k})ds+\int_{0}^te^{A(t-s)}b(s)dB_s^{H},
\end{equation}
where we let $b(s):=b_1(s)$ and $B_s^{H}:=B_1^{H}(s)$. Without causing confusion, we will use this simple notation in the following content.

Before proving Theorem \ref{sec1-thm.1}, we also need to prove some technical lemmas, which will make the proof of our main result presented in a more convenient form.
\begin{lemma}\label{sec3-lem.1}
Assume the condition (A1) is satisfied and $\mathbb{E}|u_0|^p<\infty$. For  any integer $p\geq1$, there exist constants $C_{1,p,T}, C_{2,p,T}$ depending on $p$ and $T$, such that
\begin{equation}\label{sec3-eq.ut}
\mathbb{E}\left(\sup_{t_0\leq t\leq T}|U_t|^p\right)<C_{1,p,T}
\end{equation}
and
\begin{equation}\label{sec3-eq.vt}
\mathbb{E}\left(\sup_{t_0\leq t\leq T}|V_t|^p\right)<C_{2,p,T}.
\end{equation}
\end{lemma}
\begin{proof}
For \eqref{sec3-eq.ut}, we use basic inequality for \eqref{sec3-eq.1},
$$
|U_t|^p\leq c_p|e^{A(t-t_0)}u_0|^p+c_p\left|\int_{t_0}^te^{A(t-s)} f(U_s)ds\right|^p+c_p\left|\int_{t_0}^te^{A(t-s)}b(s)dB_s^H\right|^p.
$$
Then using condition (A1) and the Kahane-Khintchine formula, we obtain
\begin{align*}
\mathbb{E}\left(\sup_{t_0\leq t\leq T}|U_t|^p\right)&\leq c_p\sup_{t}e^{p\mu[A](t-t_0)}\mathbb{E}|u_0|^p+c_{T,p}\sup_{t}e^{p\mu[A](t-t_0)}\int_{t_0}^t\mathbb{E}\sup_s|U_s|^pds\\
&\qquad+c_p\left(\mathbb{E}\left|\int_{t_0}^te^{A(t-s)}b(s)dB_s^H\right|^2\right)^{p/2}\\
&\leq c_p\sup_{t}e^{p\mu[A](t-t_0)}\mathbb{E}|u_0|^p+c_{T,p}\sup_{t}e^{p\mu[A](t-t_0)}\int_{t_0}^t\mathbb{E}\sup_s|U_s|^pds\\
&\qquad+c_p\left(\int_{t_0}^t\int_{t_0}^t|e^{A(t-u)}b_i(u)||e^{A(t-v)}b(v)||u-v|^{2H-2}dudv\right)^{p/2}\\
&\leq c_p\sup_{t}e^{p\mu[A](t-t_0)}\left(\mathbb{E}|u_0|^p+\int_{t_0}^t\mathbb{E}\sup_s|U_s|^pds+(t-t_0)^{2H}\right).
\end{align*}
Thus, the desired result follows from the Gronwall inequality.

For the proof of \eqref{sec3-eq.vt},  we can use the similar method as in \eqref{sec3-eq.ut} to study \eqref{sec3-eq.2} ,
\begin{align}\label{sec3-eq.vt1}
\mathbb{E}\left(\sup_t|V_t|^p\right)
&\leq c_p\sup_{t}e^{p\mu[A](t-t_0)}\left(\mathbb{E}|u_0|^p+\sum_{j=0}^{k-1}h_j\mathbb{E}|V_j|^p+(t-t_k)\mathbb{E}|V_k|^p+(t-t_0)^{2H}\right).
\end{align}
Therefore, according to the recursive relationship
\begin{align}\label{sec3-eq.vt2}
\mathbb{E}|V_k|^p
&\leq c_{p,T}\left(1+\sum_{j=0}^{k-1}h_j\mathbb{E}|V_j|^p\right), ~~k=1,2,\cdots,n,
\end{align}
and
\begin{align}\label{sec3-eq.vt3}
\mathbb{E}|V_0|^p=\mathbb{E}|u_0|^p<\infty.
\end{align}
Together \eqref{sec3-eq.vt1}--\eqref{sec3-eq.vt3},  we obtain the desired result.
\end{proof}

\begin{lemma}\label{sec3-lem.2}
Assume the conditions (A1)-(A3) are satisfied and $\mathbb{E}|u_0|^p<\infty$, then there exist constants $C_{3,p,T}, C_{4,p,T}$ depending on $p$ and $T$, such that
\begin{align}\label{sec3-eq.DrU}
\mathbb{E}\left[\sup_{t_0\leq r, t\leq T}|D_rU_t|^p\right]\leq C_{3,p,T}
\end{align}
and
\begin{align}\label{sec3-eq.DrrU}
\mathbb{E}\left[\sup_{t_0\leq r,r' t\leq T}|D^2_{r,r'}U_t|^p\right]\leq C_{4,p,T}.
\end{align}
\end{lemma}
\begin{proof}
Let $r\in[t_0,T]$. Taking the Malliavin derivative $D_r$ on both sides of \eqref{sec3-eq.1} leads to
\begin{align}\label{sec3-eq.DrU-1}
D_r U_t=\int_r^te^{A(t-s)}\partial f(U_s)D_rU_sds+e^{A(t-r)}b(r).
\end{align}
Then by conditions (A2)--(A3) and inequality \eqref{sec3-eq.ut}, we have
\begin{align*}
\mathbb{E}\left(\sup_{r,t}|D_r U_t|^p\right)&\leq c_{p,t}\sup_se^{\mu[A](t-s)}\left(1+\mathbb{E}(\sup_s|U_s|^{p\nu})\right)\int_r^t\mathbb{E}\left(\sup_{r,s}|D_r U_s|^p\right)ds\\
&\quad\qquad\qquad+c_{p,t}\sup_se^{\mu[A](t-s)}\\
&\leq c_{p,T}\left(\int_r^t\mathbb{E}\left(\sup_{r,s}|D_r U_s|^p\right)ds+1\right).
\end{align*}
Thus we get \eqref{sec3-eq.DrU} by Gronwall inequality.

For the proof of \eqref{sec3-eq.DrrU}. Differentiating \eqref{sec3-eq.DrU-1}, we have for $r\vee r'<t$,
\begin{align*}
D_{r,r'}^2 U_t&=\int_{r\vee r'}^te^{A(t-s)}\partial f(U_s)D^2_{r,r'}U_sds+\int_{r\vee r'}^te^{A(t-s)}\langle\partial^2 f(U_s),D_{r}U_s\otimes D_{r'}U_s\rangle ds.
\end{align*}
By condition (A2) and \eqref{sec3-eq.ut}, we can see
$$\mathbb{E}\left(\sup_s|e^{A(t-s)}\partial f(U_s)+e^{A(t-s)}\partial^2 f(U_s)|^p\right)<c_{p,T},$$
and by the result in \eqref{sec3-eq.DrU},
$$\mathbb{E}\left(\sup_{r,r',s}|D_{r}U_s\otimes D_{r'}U_s|^{p}\right)\leq \mathbb{E}\left(\sup_{r,s}|D_{r}U_s|^{2p}\right)^{1/2}\mathbb{E}\left(\sup_{r',s}| D_{r'}U_s|^{2p}\right)^{1/2}<c_{p,T}.$$
This gives
\begin{align*}
\mathbb{E}\left(\sup_{r,r',t}|D_{r,r'}^2 U_t|^p\right)&\leq c_{p,T}+c_{p,T}\int_{r\vee r'}^t\mathbb{E}\left(\sup_{r,r',s}|D_{r,r'}^2 U_s|^p\right)ds.
\end{align*}
Then the desired result follows from the Gronwall inequality.
\end{proof}

\begin{lemma}\label{sec3-lem.3}
Assume the conditions (A1)-(A3) are satisfied and $\mathbb{E}|u_0|^p<\infty$, then there exist constants $C_{5,p,T}, C_{6,p,T}$ depending on $p$ and $T$, such that
\begin{align}\label{sec3-eq.DrV}
\mathbb{E}\left[\sup_{t_0\leq r, t\leq T}|D_rV_t|^p\right]\leq C_{5,p,T}
\end{align}
and
\begin{align}\label{sec3-eq.DrrV}
\mathbb{E}\left[\sup_{t_0\leq r,r', t\leq T}|D^2_{r,r'}V_t|^p\right]\leq C_{6,p,T}.
\end{align}
\end{lemma}
\begin{proof}
According to \eqref{sec3-eq.2}, we consider the Malliavin derivative $D_rV_t$. It is easy to find that,
for$t_k<r<t$,
\begin{align}\label{sec3-eq.DrV1}
D_rV_t=e^{A(t-r)}b(r).
\end{align}
If $t_{k-1}<r<t_{k}$, we can see
\begin{align}\label{sec3-eq.DrV2}
D_rV_t=\int_{t_k}^te^{A(t-s)}ds\partial f(V_k)D_rV_k+e^{A(t-r)}b(r).
\end{align}
If  $t_{k-2}<r<t_{k-1}$,
\begin{align}\label{sec3-eq.DrV3}
D_rV_t=\int_{t_{k-1}}^{t_k}e^{A(t-s)}ds\partial f(V_{k-1})D_rV_{k-1}+\int_{t_k}^te^{A(t-s)}ds\partial f(V_k)D_rV_k+e^{A(t-r)}b(r).
\end{align}

Similarly, for $t_{0}<r<t_{1}$,
\begin{align}\label{sec3-eq.DrV4}
D_rV_t=\int_{t_1}^{t_2}e^{A(t-s)}ds\partial f(V_{1})D_rV_{1}+\cdots+\int_{t_k}^te^{A(t-s)}ds\partial f(V_k)D_rV_k+e^{A(t-r)}b(r).
\end{align}
Thus, to prove \eqref{sec3-eq.DrV}, we only need to use induction to calculate the upper bound of $D_rV_j, j=0,1,\ldots,k$ sequentially.

Note that,  for $t_0<r<t_1$, $D_rV_0=0$ and $D_rV_1=e^{A(t_1-r)}b(r)$. Then
$$\mathbb{E}\left(\sup_r|D_rV_1|^p\right)\leq C_{p,T}.$$
Then by conditions (A2)--(A3) and \eqref{sec3-eq.vt}, together \eqref{sec3-eq.DrV1}--\eqref{sec3-eq.DrV4}, we have the inequality \eqref{sec3-eq.DrV}.

For the inequality \eqref{sec3-eq.DrrV},  we only need to use the methods similar to \eqref{sec3-eq.DrrU} and \eqref{sec3-eq.DrV}. This completes the proof.
\end{proof}

\begin{lemma}\label{sec3-lem.4}
Assume the conditions (A1)-(A3) are satisfied and $\mathbb{E}|u_0|^p<\infty$, then there exist constants $C_{7,p,T}, C_{8,p,T}$ depending on $p$ and $T$, such that
\begin{align}\label{sec3-eq.Drint}
\mathbb{E}\left[\sup_{t_0\leq r, t\leq T}|D_r\left(\int_0^1\partial f[\theta U_{t_j}+(1-\theta)V_j]d\theta\right)|^p\right]\leq C_{7,p,T}
\end{align}
and
\begin{align}\label{sec3-eq.Drrint}
\mathbb{E}\left[\sup_{t_0\leq r,r', t\leq T}|D^2_{r,r'}\left(\int_0^1\partial f[\theta U_{t_j}+(1-\theta)V_j]d\theta\right)|^p\right]\leq C_{8,p,T}.
\end{align}
\end{lemma}
\begin{proof}
The proof of this lemma is similar to the proof of Lemma \ref{sec3-lem.3}, we need to  rewrite $D_r\left(\partial f[\theta U_{t_j}+(1-\theta)V_j]\right)$ as
$$
D_r\left(\partial f[\theta U_{t_j}+(1-\theta)V_j]\right)=\partial^2 f[\theta U_{t_j}+(1-\theta)V_j]\left(\theta D_rU_{t_j}+(1-\theta)D_rV_j\right).
$$
Then based on the polynomial growth condition of the second derivative of function $f$ in assumption (A2) and Lemmas \ref{sec3-lem.1}--\ref{sec3-lem.4}, we obtain \eqref{sec3-eq.Drint}.

By the same way,
\begin{align*}
&D^2_{r,r'}\left(\partial f[\theta U_{t_j}+(1-\theta)V_j]\right)\\
&=D_{r'}\left(\partial^2 f[\theta U_{t_j}+(1-\theta)V_j]\left(\theta D_rU_{t_j}+(1-\theta)D_rV_j\right)\right)\\
&=\partial^3 f[\theta U_{t_j}+(1-\theta)V_j]\left(\theta D_rU_{t_j}+(1-\theta)D_rV_j\right)\left(\theta D_{r'}U_{t_j}+(1-\theta)D_{r'}V_j\right)\\
&\qquad+\partial^2 f[\theta U_{t_j}+(1-\theta)V_j]\left(\theta D^2_{r,r'}U_{t_j}+(1-\theta)D^2_{r,r'}V_j\right).
\end{align*}
Then \eqref{sec3-eq.Drrint} follows from the polynomial growth condition of the third derivative of function $f$ and Lemmas \ref{sec3-lem.1}--\ref{sec3-lem.4}.
\end{proof}

With the four technical lemmas mentioned above, we can consider proving the main conclusion of this paper.\\
\textbf{Proof of Theorem \ref{sec1-thm.1}.} The proof is divided into fore steps.

\textbf{Step 1: A representation for the error process \eqref{sec3-eq.Z}. }
Denote by the error process $Z_t:=U_t-V_t$, and for convenience we will also denote by $Z_k:=Z_{t_k}$. Then
$$
Z_t=e^{A(t-t_k)}Z_{t_k}+\int_{t_k}^te^{A(t-s)}[f(U_s)-f(V_k)]ds.
$$
For the integral term,
\begin{align*}
\int_{t_k}^te^{A(t-s)}[f(U_s)-f(V_k)]ds
=\int_{t_k}^te^{A(t-s)}[f(U_s)-f(U_{t_k})]ds+\int_{t_k}^te^{A(t-s)}[f(U_{t_k})-f(V_k)]ds,
\end{align*}
where the second integral on the right side of the above equation
\begin{align*}
\int_{t_k}^te^{A(t-s)}[f(U_{t_k})-f(V_k)]ds
&=\int_{t_k}^te^{A(t-s)}ds(U_{t_k}-V_k)\int_0^1\partial f[\theta U_{t_k}+(1-\theta)V_k]d\theta\\
&=A^{-1}(e^{A(t-t_k)}-Id)Z_k\int_0^1\partial f[\theta U_{t_k}+(1-\theta)V_k]d\theta.
\end{align*}
Thus, we can rewrite $Z_t$ as follows,
\begin{align*}
Z_t=:Q_k(t)Z_{t_k}+R_k(t),
\end{align*}
where
$$
R_k(t)=\int_{t_k}^te^{A(t-s)}[f(U_s)-f(U_{t_k})]ds
$$
and
$$Q_k(t)=e^{A(t-t_k)}+A^{-1}(e^{A(t-t_k)}-Id)\int_0^1\partial f[\theta U_{t_k}+(1-\theta)V_k]d\theta.$$

By assumption (A4), we obtain that
$$
|Q_k(t)|\leq \sup_{t}e^{\mu[A]t}+L(t-t_k)[1+\sup_t|U_{t}|+\sup_t|V_t|].
$$

If  $t=t_{k+1}$,
\begin{align}\label{sec3-eq.Z1}
Z_{k+1}&=Q_{k}(k+1)Z_k+R_{k}(k+1)\nonumber\\
&=Q_{k}(k+1)\Big(Q_{k-1}(k)Z_{k-1}+R_{k-1}(k)\Big)+R_{k}(k+1)\nonumber\\
&\cdots\nonumber\\
&=\prod_{j=0}^kQ_{j}(j+1)Z_0+\sum_{j=0}^{k}\prod_{\ell=j+1}^kQ_{\ell}(\ell+1)R_{j}(j+1),
\end{align}
where we let $\prod_{\ell=k+1}^kQ_{\ell}(\ell+1)=1$ by convention and
$$
R_j(j+1)=\int_{t_j}^{t_{j+1}}e^{A(t_{j+1}-s)}[f(U_s)-f(U_{t_j})]ds
$$
and
$$Q_j(j+1)=e^{A(t_{j+1}-t_j)}+A^{-1}(e^{A(t_{j+1}-t_j)}-Id)\int_0^1\partial f[\theta U_{t_j}+(1-\theta)V_j]d\theta.$$

Note that $Z_0=U_0-V_0=0$. Then for all $j=0,1,\ldots,k$,
\begin{align}\label{sec3-eq.Q}
\mathbb{E}\sup_j|Q_j(j+1)|^p\leq C.
\end{align}

If $t\in(t_k, t_{k+1})$, and $Z_0=0$, similar to \eqref{sec3-eq.Z1}, we have
\begin{align}\label{sec3-eq.Z}
Z_{t}=\sum_{j=0}^k\prod_{\ell=j+1}^kQ_{\ell}^tR_{j}^t=\sum_{j=0}^kM_j^tR_{j}^t,
\end{align}
where $\prod_{\ell=j+1}^kQ_{\ell}^t=:M_j^t$,
\begin{align*}
R_{j}^t&=\int_{t_j}^{t_{j+1}\wedge t}e^{A(t_{j+1}\wedge t-s)}[f(U_s)-f(U_{t_{j}})]ds\\
&=\int_{t_j}^{t_{j+1}\wedge t}e^{A(t_{j+1}\wedge t-s)}\int_0^1\partial f[\theta U_s+(1-\theta)U_{t_{j}}]d\theta \Delta U_{t_j,s}ds\\
&=:\int_{t_j}^{t_{j+1}\wedge t}e^{A(t_{j+1}\wedge t-s)}f_1(j,s)\Delta U_{t_j,s}ds
\end{align*}
and
$$Q_{j}^t=e^{A(t_{j+1}\wedge t-t_j)}+A^{-1}(e^{A(t_{j+1}\wedge t-t_j)}-Id)\int_0^1\partial f[\theta U_{t_j}+(1-\theta)V_j]d\theta.$$

\textbf{Step 2: Estimate of $M_j^t$. } If we use the inequality \eqref{sec3-eq.Q} to estimate of $M_j^t$ may not be optimal rate. To obtain the optimal rate estimate, we need to bound the Malliavin derivative of $M_j^t$. A straightforward computation for $D_r M_j^t$ and $D^2_{r,r'} M_j^t$ yields
\begin{align}\label{sec3-eq.DrM}
D_r M_j^t=\sum_{\ell=j+1}^k\left(Q_{j+1}^t\cdots Q_{\ell-1}^t D_rQ_{\ell}^tQ_{\ell+1}^t\cdots Q_{k}^t\right)
\end{align}
and
\begin{align}\label{sec3-eq.DrrM}
D^2_{r,r'} M_j^t&=\sum_{\ell=j+1}^k\left(Q_{j+1}^t\cdots Q_{\ell-1}^t D^2_{r,r'}Q_{\ell}^tQ_{\ell+1}^t\cdots Q_{k}^t\right)\nonumber\\
&\qquad+\sum_{\ell\neq\ell^{'}, \ell,\ell^{'}=j+1}^k\left(Q_{j+1}^t\cdots D_rQ_{\ell}^t\cdots D_{r'}Q_{\ell^{'}}^t\cdots Q_{k}^t\right).
\end{align}

By Lemma \ref{sec3-lem.4}, \eqref{sec3-eq.Q}, \eqref{sec3-eq.DrM} and \eqref{sec3-eq.DrrM}, we can see

\begin{align}\label{sec3-eq.supDrM}
\mathbb{E}\left[\sup_{t_0\leq r, t\leq T}|D_rM_j^t|^p\right]\leq C_{p,T}
\end{align}
and
\begin{align}\label{sec3-eq.supDrrM}
\mathbb{E}\left[\sup_{t_0\leq r,r', t\leq T}|D^2_{r,r'}M_j^t|^p\right]\leq C_{p,T}.
\end{align}

\textbf{Step 3:  A decomposition for the error process. } Recall the representation \eqref{sec3-eq.Z}, then we have
\begin{align}\label{sec3-eq.ZI}
|Z_t|^2\leq C_t\sum_{j,j'=0}^kM_j^tR_j^tM_{j'}^tR_{j'}^t=:C_t\sum_{j,j'=0}^kI_{j,j'},
\end{align}
where $C_t$  contains the maximum value $\sup_s e^{\mu[A](t_{j+1}\wedge t-s)}$ of the integrand function in the expression of $R_j^t$, and
\begin{align*}
I_{j,j'}= \int_{t_j}^{t_{j+1}\wedge t}\int_{t_{j'}}^{t_{j'+1}\wedge t}M_j^tf_1(j,s)\Delta U_{t_j,s}
\cdot M_{j'}^tf_1(j',s')\Delta U_{t_{j'},s'}dsds'.
\end{align*}

Since
\begin{align*}
\Delta U_{t_j,s}&=e^{A(s-t_0)}u_0+\int_{t_0}^se^{A(s-r)}f(U_r)dr+\int_{t_0}^se^{A(s-r)}b(r)dB_r^{H}\\
&\quad-e^{A(t_j-t_0)}u_0+\int_{t_0}^{t_j}e^{A(t_j-r)}f(U_r)dr+\int_{t_0}^{t_j}e^{A(t_j-r)}b(r)dB_r^{H}\\
&=\left(e^{A(s-t_0)}-e^{A(t_j-t_0)}\right)u_0+\int_{t_0}^{t_j}\left(e^{A(s-r)}-e^{A(t_j-r)}\right)f(U_r)dr+\int_{t_j}^se^{A(s-r)}f(U_r)dr\\
&\quad+\int_{t_0}^{t_j}\left(e^{A(s-r)}-e^{A(t_j-r)}\right)b(r)dB_r^{H}+\int_{t_j}^se^{A(s-r)}b(r)dB_r^{H}\\
&=:\Lambda^1_{t_j,s}+\Lambda^2_{t_j,s}+\Lambda^3_{t_j,s}+\Lambda^4_{t_j,s}+\Lambda^5_{t_j,s}.
\end{align*}
So we can write
\begin{align}\label{sec3-eq.IJJ}
I_{j,j'}&= \int_{t_j}^{t_{j+1}\wedge t}\int_{t_{j'}}^{t_{j'+1}\wedge t}M_j^tf_1(j,s)\cdot M_{j'}^tf_1(j',s')
\left(\sum_{\ell=1}^5\Lambda^1_{t_j,s}\right)\left(\sum_{\ell'=1}^5\Lambda^1_{t_{j'},s'}\right)dsds'.
\end{align}

\textbf{Step 3:  Estimate of $I_{j,j'}$. } Note that $\mathbb{E}\left|M_j^tf_1(j,s)\cdot M_{j'}^tf_1(j',s')\right|^p<\infty$, since \eqref{sec3-eq.Q}. By the proof of Theorem 3.1 in \cite{Kam24}, for any $\varepsilon>0$, we can find
\begin{align*}
\mathbb{E}\left|\Lambda^1_{t_j,s}+\Lambda^2_{t_j,s}+\Lambda^3_{t_j,s}\right|^p\leq c |s-t_j|^{p(1-\varepsilon)}
\end{align*}
holds under the condition $\mathbb{E}|Au_0|^p<\infty$ and assumption (A4).

This gives
\begin{align}\label{sec3-eq.13}
&\left|\mathbb{E}\left[M_j^tf_1(j,s)\cdot M_{j'}^tf_1(j',s')\Lambda^k_{t_j,s}\Lambda^{k'}_{t_{j'},s'}\right]\right|\nonumber\\
&\leq\left(\mathbb{E}\left|M_j^tf_1(j,s)\cdot M_{j'}^tf_1(j',s')\right|^3\right)^{1/3}
\left(\mathbb{E}\left|\Lambda^k_{t_j,s}\right|^3\right)^{1/3}
\left(\mathbb{E}\left|\Lambda^{k'}_{t_{j'},s'}\right|^3\right)^{1/3}\nonumber\\
&\leq c |s-t_j|^{1-\varepsilon}|s'-t_j'|^{1-\varepsilon}, ~~\text{for}~k,k'=1,2,3.
\end{align}

For $k,k'=4,5$, similar to \eqref{sec3-eq.supDrM} and \eqref{sec3-eq.supDrrM}, we can obtain
$$\mathbb{E}\left[\sup_{t_0\leq r, t\leq T}|D_{r}(M_j^tf_1(j,s)\cdot M_{j'}^tf_1(j',s'))|^p\right]\leq C_{p,T}$$
and
$$\mathbb{E}\left[\sup_{t_0\leq r,r', t\leq T}|D^2_{r,r'}(M_j^tf_1(j,s)\cdot M_{j'}^tf_1(j',s'))|^p\right]\leq C_{p,T}.$$
This satisfy the conditions in Lemma \ref{sec2-lem-3}. Therefore, we have the following conclusion
\begin{align}\label{sec3-eq.45}
\left|\mathbb{E}\left[M_j^tf_1(j,s)\cdot M_{j'}^tf_1(j',s')\Lambda^k_{t_j,s}\Lambda^{k'}_{t_{j'},s'}\right]\right|\leq c |s-t_j||s'-t_j'|+c\langle \mathbf{1}_{[t_j,s]}, \mathbf{1}_{[t_j',s']}\rangle_{\mathfrak{H}}.
\end{align}

If one of $k$ and $k'$ belongs to $\{1,2,3\}$ and the other belongs to $\{4,5\}$. Based on symmetry, we only need to consider the case of $k=1,2,3$ and $k'=4,5.$
By H\"{o}lder inequality, \eqref{sec3-eq.13} and \eqref{sec3-eq.45},
\begin{align}\label{sec3-eq.13-45}
&\left|\mathbb{E}\left[M_j^tf_1(j,s)\cdot M_{j'}^tf_1(j',s')\Lambda^k_{t_j,s}\Lambda^{k'}_{t_{j'},s'}\right]\right|\nonumber\\
&\leq\left(\mathbb{E}\left|M_j^tf_1(j,s)\Lambda^k_{t_j,s}\right|^2\right)^{1/2}
\left(\mathbb{E}\left|M_{j'}^tf_1(j',s')\Lambda^{k'}_{t_{j'},s'}\right|^2\right)^{1/2}\nonumber\\
&\leq c |s-t_j|^{1-\varepsilon}\left(|s'-t_j'|+|s'-t_j'|^H\right).
\end{align}

Together \eqref{sec3-eq.IJJ}--\eqref{sec3-eq.13-45}, we have
\begin{align}\label{sec3-eq.IJJ2}
\mathbb{E}|I_{j,j'}|\leq C h_j^{2-\varepsilon} h_{j'}^{2-\varepsilon}+C h_jh_{j'}\langle \mathbf{1}_{[t_j,t_{j+1}\wedge t]}, \mathbf{1}_{[t_j',t_{j'+1}\wedge t]}\rangle_{\mathfrak{H}}
+h_j^{2-\varepsilon}h_{j'}^{1+H},
\end{align}
where $h_j=t_{j+1}-t_j$.

\textbf{Step 4:  Conclusion. }
Plugging \eqref{sec3-eq.IJJ2} into \eqref{sec3-eq.ZI}, we conclude that
\begin{align*}
\sup_t\mathbb{E}|Z_t|^2&\leq C \max_jh_j \max_{j'}h_{j'}\left(\sum_{j,j'=0}^kh_j^{1-\varepsilon} h_{j'}^{1-\varepsilon}+ \sum_{j,j'=0}^kh_j^{H} h_{j'}^{H}+\sum_{j,j'=0}^kh_j^{1-\varepsilon} h_{j'}^{H}\right)\\
&\leq C h_{max}^2,
\end{align*}
where $h_{max}=\max_jh_j$, we use
\begin{align*}
\langle \mathbf{1}_{[t_j,t_{j+1}\wedge t]}, \mathbf{1}_{[t_j',t_{j'+1}\wedge t]}\rangle_{\mathfrak{H}}
&=\frac12[|t_{j+1}\wedge t-t_{j'}|^{2H}+|t_{j'+1}\wedge t-t_{j}|^{2H}\\
&\qquad\qquad-|t_{j'}-t_j|^{2H}-|t_{j+1}\wedge t-t_{j'+1}\wedge t|^{2H}]\\
&\leq c h_j^{H} h_{j'}^{H}, ~~~\text{see Lemma 2.4 in \cite{SXY}}
\end{align*}
in the first inequality and use $\sum_{j}h_j^{a}<\infty$ for $a<1$ in the last inequality.

Therefore, we obtain the convergence rate
$$
\sup_{k=0,1,\cdots,N}\sqrt{\mathbb{E}|U_{t_k}-V_{k}|^2}\leq C \max_k\,h_k.
$$
This completes the proof. $\hfill\blacksquare$

\section{Numerical experiment}\label{sec4}

We conclude this paper with a numerical example and verify our statement about the optimality of rate.  In \cite{Kam24}, Kamrani \emph{et al}. consider the following stiff nonlinear system
$$dU_t=(EU_t+\sin(U_t))dt+dB_t^H, ~t\in[0.1],$$
$$U_0=\sqrt{\frac2{m+1}}(\sin\frac{\pi}{m+1},\sin\frac{2\pi}{m+1},\cdots,\sin\frac{m\pi}{m+1})^{\top},$$
where $B^H$ is an $m$-dimensional fBm and $E\in\mathbb{R}^{m\times m}$ given by
\begin{equation*}
E=(m+1)^2
\left[
\begin{array}{cccccc}
-2 & 1 & 0 & 0 & \cdots &0\\
1 & -2 & 1 & 0 & \cdots &0\\
0 & \ddots & \ddots & \ddots & \ddots &\vdots\\
\vdots & \ddots & \ddots & \ddots & \ddots &\vdots\\
0 & \cdots & 0 & 1 & -2 &1\\
0 & \cdots & \cdots & 0 & 1 &-2
\end{array}
\right]
\end{equation*}
Their results show that the optimal rate of convergence of one, which is consistent with our proof.

Moreover, we can also provide a numerical example of a special system to further verify the optimal rate.  In the end of this paper we consider the following special system
$$\text{d}U_t=(-2U_t-\sin(U_t))\text{d}t+\text{d}B_t^{H,1}, U_0=1, t\in [0,1],$$
where $B^{H,1}$ is a one-dimensional fBm with Hurst parameter $1/2<H<1$.

To find the convergence rate of the Euler method, in our numerical computation we take the time step sizes $2^{-k}, k=1,2,3,4,5$.  As in the previous example the numerical solution with a much smaller step size (in this example we take the
step size $2^{-11}$) is used to represent the reference solution. We perform $N=1000$ simulations, and we compute the mean square error by
$$\varepsilon=\Big(\frac1N\sum_{j=1}^N(U_{t_j}-V_j)^2\Big)^{1/2},$$
where $U_{t_j}$ denotes by the exact solution, and $V_j$ denotes by the numerical solution at step size.

\begin{figure}[htbp]
\begin{minipage}{0.48\linewidth}
\vspace{3pt}
\centerline{\includegraphics[width=\textwidth]{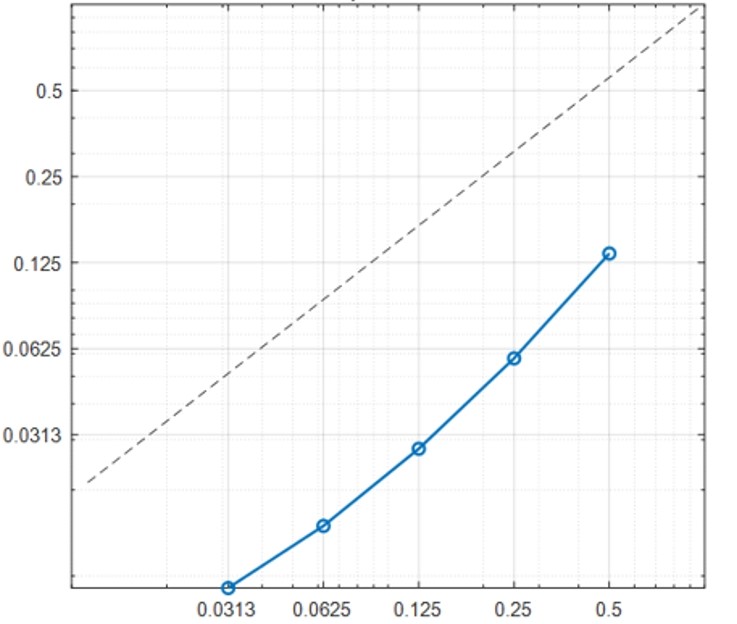}}
 \caption{$H=0.6$}
 \vspace{3pt}
\centerline{\includegraphics[width=\textwidth]{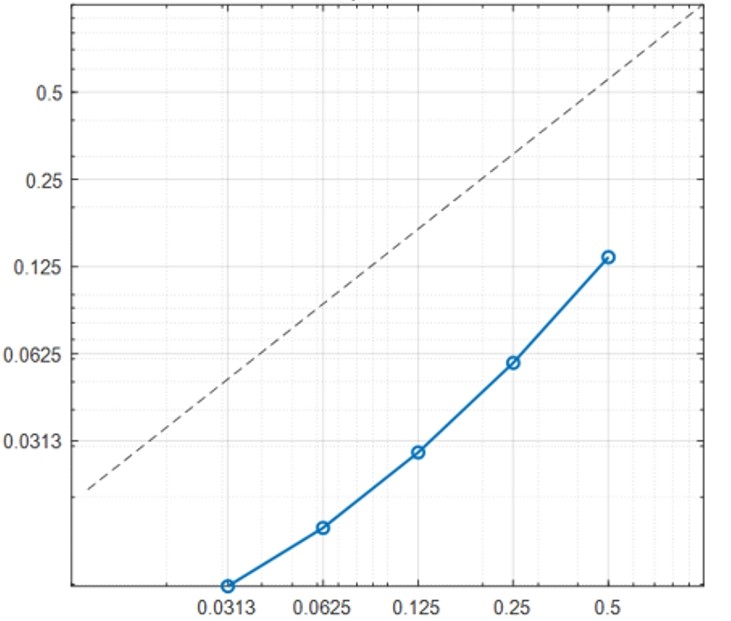}}
 \caption{$H=0.7$}
 \vspace{3pt}
  \end{minipage}
  \begin{minipage}{0.48\linewidth}
\vspace{3pt}
\centerline{\includegraphics[width=\textwidth]{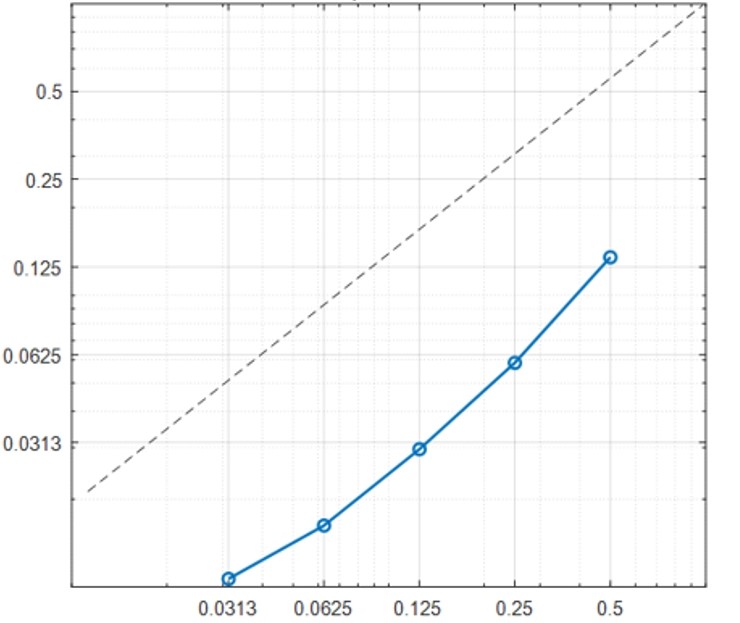}}
 \caption{$H=0.8$}
 \vspace{3pt}
\centerline{\includegraphics[width=\textwidth]{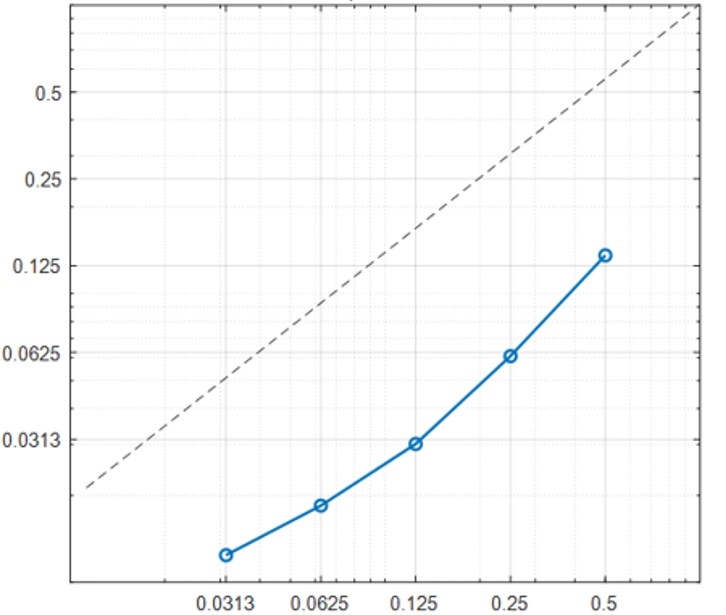}}
 \caption{$H=0.9$}
 \vspace{3pt}
  \end{minipage}
\end{figure}

Figures 1--4 show the root mean square errors obtained by fBm paths over $[0,1]$ for $H=0.6, 0.7, 0.8, 0.9$, respectively.
From these results, the orders of convergence of the method can be approximated as slopes of the regression lines in Figures for the different $H$-values. The results indicate an order of convergence of one.
Thus, we numerically verify that our theoretical results are in accordance with numerical results.

The optimal rate here gives us additional thinking.  We can consider the asymptotic error distribution of the exponential Euler scheme. For convenience, we consider the equidistant discretization of interval $[t_0,T]$, i.e.
$\pi: t_0<t_1<\cdots<t_N=T$ and $h_i=t_{i+1}-t_i=1/N, 0\leq i\leq N-1$. Then, under the representation of error process \eqref{sec3-eq.Z},
we can study the limit distribution of
\begin{align*}
NZ_{t}=\sum_{j=0}^kM_j^t(NR_{j}^t), \quad \text{as} ~N\to\infty.
\end{align*}
This requires further calculation of the precise convergence rate of $R_{j}^t$.
Note that
\begin{align*}
R_{j}^t&=\int_{t_j}^{t_{j+1}\wedge t}e^{A(t_{j+1}\wedge t-s)}f_1(j,s)\Delta U_{t_j,s}ds,
\end{align*}
where the expansion representation of $\Delta U_{t_j,s}$ will generate additional matrix $A$. The stiff term $|A|t\gg 1, t\in[t_0,T]$ will be the biggest trouble,
although the inequality $Ae^{A(t_{j+1}-s)}\leq c|t_{j+1}-s|^{-1}$ holds by assumption (A4).
Therefore, characterizing the limits of $NR_{j}^t, N\to\infty$ clearly and studying the asymptotic distribution of error process is our future aim.

\bigskip

\textbf{Data Availability} ~Data sharing not applicable to this article as no datasets were generated or analyzed during
the current study.

\textbf{Declaration of interests} ~The authors declare that they have no known competing financial interests or personal relationships that
could have appeared to influence the work reported in this paper.


\end{document}